\documentclass[12pt, reqno]{amsart}
\usepackage{amsmath, cancel, amsthm, amscd, amsfonts, amssymb, graphicx, color}
\usepackage[bookmarksnumbered, colorlinks, plainpages]{hyperref}
\hypersetup{colorlinks=true,linkcolor=red, anchorcolor=green,
citecolor=cyan, urlcolor=red, filecolor=magenta, pdftoolbar=true}
\usepackage{ulem}

\newcommand{\cs}{\ensuremath{C^\ast}}

\newcommand{\vf}{\varphi}
\newcommand{\la}{\left\langle}
\newcommand{\NN}{\mathcal{N}}
\newcommand{\ra}{\right\rangle}

\newcommand{\pa}{{\mathcal P}(\A)}
\newcommand{\pM}{{\mathcal P}(\vm)}

\newcommand{\vr}{{\mathcal M}}

\newcommand{\pma}{{\mathcal P_A}(\vm)}

\newcommand{\sm}{{\mathcal S}(\vm)}

\newcommand{\C}{\mathbb{C}}

\newcommand{\Z}{\mathbb{Z}}
\newcommand{\T}{\mathbb{T}}

\newcommand{\N}{\mathbb{N}}
\newcommand{\R}{\mathbb{R}}
\newcommand{\RR}{\mathcal{R}}

\newcommand{\cl}{{\rm cl}}

\newcommand{\h}{{\mathcal H}}

\newcommand{\M}{{\mathcal M}}

\newcommand{\lh}{{\mathcal B}({\mathcal H})}

\newcommand{\A}{\mathfrak{A}}
\newcommand{\vm}{\mathfrak{M}}
\newcommand{\vn}{\mathfrak{N}}

\newcommand{\B}{\mathfrak{B}}

\newcommand{\mnc}{{\mathcal M}_n(\C)}

\newcommand{\soc}{{\rm soc}}

\newcommand{\sga}{\mathcal{S}_A(\A)}
\newcommand{\sgm}{\mathcal{S}_A(\vm)}

\newcommand{\tr}{{\rm Tr}}

\newcommand{\uno}{{\bf\rm \textbf{1}}}

\DeclareMathSymbol{\subsetneqq}{\mathbin}{AMSb}{36}
\newtheorem{th1}{{\bfTheorem}}[section]
\newtheorem{thm}[th1]{{\bf Theorem}}
\newtheorem{lem}[th1]{{\bf Lemma}}
\newtheorem{lemma}[th1]{{\bf Lemma}}
\newtheorem{prop}[th1]{{\bf Proposition}}
\newtheorem{cor}[th1]{{\bf Corollary}}
\theoremstyle{remark}
\newtheorem{rem}[th1]{Remark}
\newtheorem{example}[th1]{Example}
\newtheorem{rems}[th1]{Remarks}
\theoremstyle{definition}

\newtheorem{qu}{Question}

\newcommand{\lho}{\mathcal{B}_{A^{1/2}}(\h)}
\usepackage{color}

\begin{document}

\setcounter{page}{1}
\title[On the   $A$- spectrum
  for $A$-bounded operators on von-Neumann algebras]{On the $A$-spectrum
  for $A$-bounded operators on von-Neumann algebras}

\author[H. Baklouti \MakeLowercase{and} K. Dhifaoui \MakeLowercase{and} M.~Mabrouk]
{H. Baklouti$^{1}$,   K. Dhifaoui$^{1}$ \MakeLowercase{and} M.~Mabrouk$^{1}$}

\address{  Faculty of Sciences of Sfax, Department of Mathematics University of Sfax Tunisia
}
%\email{msmabrouk@uqu.edu.sa}

\address{Faculty of Sciences of Sfax, Department of Mathematics University of Sfax Tunisia
%\&
%Department of Mathematics, Faculty of Applied Sciences, Umm Al-Qura University 21955 Makkah, Saudi Arabia
}
\email{msmabrouk@uqu.edu.sa}

%%%%%%%%%%%%%%%%%%%%%%%%%%%%%%%%%%
\begin{abstract}
Let $\vm$ be a von Neumann algebra and let $A$ be a nonzero positive element of $\vm$. By $\sigma_A(T) $ and $r_A(T)$ we denote the $A$-spectrum and
the $A$-spectral radius of $T\in\vm^A$, respectively. In this paper, we show that $\sigma(PTP, P\vm P)\subseteq \sigma_A(T)$. Sufficient conditions for the equality $\sigma_A(T)=\sigma(PTP, P\vm P)$ to be true are presented.  Also, we show that $\sigma_A(T)$ is finite for any $T\in\vm^A$ if and only if $A$ is in the socle of $\vm$. Next , we consider the relationship between elements
$S$ and $T\in\vm^A$ that satisfy one of the following two conditions: (1) $\sigma_A(SX)=\sigma_A(TX)$ for all $X\in\vm^A$, (2)
$r_A(SX)= r_A(TX)$ for all $X\in\vm^A$. Finally, a Gleason-Kahane-\.Zelazko's theorem for the $A$-spectrum is derived.% Finally, we  introduce and study the notion of the $A$-approximate point spectrum for element of $X\in\vm^A$.
\end{abstract}\maketitle
\vspace*{-0.75cm}
\section{Introduction}
For a complex Hilbert space $\h$   with inner product $\la ., .\ra$ we denote by $\lh$ the von Neumann algebra of all bounded linear operators on $\h$  with the identity
operator $I = I_\h$. Let also $\vm$ be  a von Neumann algebra acting $\h$. That is $\M$ is $*$-subalgebra of $\lh$ containing the identity operator $I$ and closed in the weak operator topology.  Denote by $\vm^+$ the cone of positive elements in $\vm$ and by $\vm'$
the dual space of $\vm$. The range, respectively the kernel of $T\in\vm$   are denoted by $\RR(T)$ respectively $\NN(T)$.  $\overline{\RR(T)}$ stands for the norm closure
of $\RR(T)$. From this point forward, we fix $A$ for positive element in $\vm$ and $P$
for the orthogonal projection onto $\overline{\RR(A)}$. Note that by properties of von Neumann algebras $P$ belongs also to $\vm$.

A linear functional $f\in\vm'$ is said to be positive, and write $f\ge 0$,  if $f(T)\ge 0$ for all $T\in\vm^+$. Note that  $f$ is positive if and only if $f$ is bounded and $\|f\|=f(I)$; see for instance \cite[Corollary 3.3.4]{murphy1990c}. By $\sgm$ we
denote the set of all positive functional $f$ on $\vm$ such that $f(A)=1$. This is nothing but the set of all  positive linear functionals on $\vm$  of the form $\frac{f}{f(A)}$ where $f\in\mathcal{S}(\vm)$ such that $f(A)\neq0$.   Here $\sm$ denotes the set of positive linear functional $f$ on $\vm$ such that $\|f\|=f(I)=1$.
 For an element $T\in\vm$, set $\|T\|_A:=\sup_{ f\in\sgm}\sqrt{f(T^*AT)}$.  Note that by \cite[Theorem 3.5]{mabrouk2020}, we have
 \begin{equation}\label{nh}
  \|T\|_A=\sup_{\|h\|_A=1}\|Th\|_A,
 \end{equation}  where $\|h\|_A:=\sqrt{\la Ah, h\ra}$ for any $h\in\h$.
  Observe that  $\|T\|_A=0$ if and only if $AT=0$ and $\|.\|_I=\|.\|$.  Further it may happen that $\|T\|_A=\infty$ for some $T\in\vm$.

  From now on we will denote
  \begin{equation}\label{setb}
  \vm^A=\left\{T\in\vm: \|T\|_A<\infty\right\}.
  \end{equation}
 Recall that,
given an element $T$ in $\vm$, we say that $S\in\vm$ is an $A$-adjoint of $T$   provided that $AS = T^*A$.   The   collection   of all operators  in $\vm$ admitting $A^{1/2}$-adjoints is denoted by $\vm^A$. We mention
here that this kind of equations can be studied by using the following
theorem   (called Douglas theorem for von Neumann algebras, see \cite{Nayak}).
\begin{thm} \label{lemDouglas}
Let $\vm$ be a von Neumann algebra and $X, Y \in \vm$.
Then the following conditions are equivalent:
\begin{itemize}
\item[(i)] $YY^*\leq \alpha XX^*$ for some $\alpha\geq0$;
\item[(ii)] there exists $Z\in \vm$ such that $Y=XZ$.
\end{itemize}
Moreover, if $X^*X=Y^*Y$, then $Z$ can be chosen to be a partial isometry with initial
projection the range projection of $Y$, and final projection as the range projection of $X$. \end{thm}

The set defined in \eqref{setb} has been studied extensively in the last decade. See for instance \cite{amz,  mabrouk2020, mz2023} and the references cited within those works.  Note that  $\vm^A$ and $\vm_A$ are two  subalgebras of $\vm$ such that $\vm_A\subseteq\vm^A$ and     $\|.\|_A$ induces a seminorm on $\vm^A$.   Further, if $\vm=\lh$, we have $\vm^A=\lho$.  Here \begin{align*}
\lho= \Big\{T\in\lho: \,
\exists c > 0; {\|Th\|}_{A}\leq c{\|h\|}_{A}, \forall h\in\mathcal{H}\Big\},
\end{align*} is the class of operators introduced and studied in \cite{arias2008metric} and \cite{arias2009lifting}.

The notions of spectrum and   spectral radius  have been extended for elements of  $\lho$ in \cite{BS}.
The study was further continued for operators in $\vm^A$ in \cite{ mz2023}   as follows. An operator
$T\in\vm^A$ is said to be left (resp. right) $A$-invertible in   $\vm^A$ if there exists $S\in\vm^A$ such that $AST=A$ (resp. $ATS=A$. The operator
$T\in\vm^A$ is said to be $A$-invertible in $\vm^A$ if it has an $A$-inverse in
$\vm^A$. That is there exists $S\in\vm^A$ such that $ATS = AST = A$ (or
equivalently $PTS = PST = P$).  The following characterization of the $A$-invertibility is quoted from   \cite{mz2023}.
 \begin{thm}
 (\cite[Theorem 2.14]{mz2023})
 \label{th1}
  An element $T\in\vm^A$ is $A$-invertible in $\vm^A$ if and only if the
following two conditions are satisfied:
\begin{itemize}
\item[(i)] there exists $c>0$ such that $\frac{1}{c}f(A)\leq f(T^*AT) \leq cf(A)$ for any $f\in\sm$,
\\
(or equivalently: $\mathrm{(i)}$' there exists $c>0$ such that $\frac{1}{c}\|h\|_A\leq \|Th\|_A \leq c\|h\|_A$ for any $h\in\h$,)
\item[(ii)] there exists $\alpha>0$ such that $A^2\leq \alpha ATT^*A$.
\end{itemize}
\end{thm}
The $A$-spectrum of $T$, denoted as $\sigma_A(T)$, is
defined as
\[\sigma_A(T) := \{\lambda\in\C: \lambda I-T \ \text{is not}\ A-\text{invertible in}\ \vm^A\}.\]
We often write $\lambda$ in place of $ I\lambda$, where $\lambda\in\C$.  Similarly we define the left $A$-spectrum $\sigma_A^l(T)$ and the right $A$-spectrum $\sigma_A^r(T)$ of $T$. Note that by \cite[Theorem 2.9]{mz2023} we have $\sigma_A(T)=\sigma_A^l(T)\cup\sigma_A^r(T)$.
%which is a nonempty compact set of $\C$.
 The $A$-spectral radius is defined by
\[r_A(T) := \sup\{|\lambda|: \lambda\in\sigma_A(T)\}.\]
When $A$ is well-supported, in \cite{mz2023} and  among other results  it was proved
 the following valuable characterizations of the $A$-spectrum
\begin{equation}\label{sp}
  \begin{split}
    \sigma_A(X)=\Big\{g(AX): g\in\sgm \ \text{and}\  g(AXY)=g(AX)g(AY), \forall Y\in\vr^A &
    \\
     \text{or} \  g(AYX)=g(AX)g(AY), \forall Y\in\vr^A\Big\}. \end{split}\end{equation}
Further,  we have
    \begin{equation}\label{sp1}
  \begin{split}
  \sigma_A(T)=\Big\{g(AT): g\in\sgm\ \text{so that}\   g(T^*AT)=|g(AT)|^2 \text{or}&\\
    g(ATT^*A)=g(AT)g(AT^*A)=|g(AT)|^2g(A^2) \ \Big\}.
  \end{split}\end{equation}
 and    $r_A(T)=\displaystyle\lim_{n\to\infty}\|T^n\|_A^{1/n}$ for every $T\in\vm^A$; see  \cite[Theorems 3.10 \& 3.14]{mz2023}.

  An important fact that can be deduced
from \eqref{sp1} is the so called $A$-Spectral Permanence theorem. That is, if $\vn$ be a von Neumann subalgebra of a $\vm$ such that $I, A\in\vn$, then elements of $\vn$ have the same $A$-spectra in $\vn$ as in $\vm$. Or, to say it differently, $T$ in $\vn$ is $A$-invertible in $\vn$ if and only if $T$ is invertible in $\vm$.

  Finally, we mention that for $A = I$, the notions of $A$-spectrum $\sigma_A(T)$ and $A$-spectral radius $r_A(T)$  become identical with spectrum $\sigma(T)$ and spectral radius $r(T)$  respectively. In particular, \eqref{sp1} writes
\begin{equation}\label{spr}
  \sigma(T)=\Big\{g(T): g\in\sm\ \text{so that}\   g(T^*T)=|g(T)|^2 \ \text{or}\
    g(TT^*)= |g(T)|^2) \ \Big\}.
  \end{equation}

Given a  unital algebra $\A$,  a linear functional $\phi: \A\longrightarrow\C$ is said to be multiplicative if
$\phi(xy) = \phi(x)\phi(y)$ for all $x, y \in\A$. A continuous multiplicative linear functional is called a character.
 %In contrast to linear functionals, t
 The existence of multiplicative linear functionals
is generally not guaranteed, but they are always present when $\A$ is a commutative
 Banach algebra. A classical result of Gleason-Kahane-\.Zelazko  (or GKZ-theorem for short) identifies
multiplicative functionals amongst the members of $\A'$ when $\A$ is a Banach algebra. It  states that if $\phi: \A\longrightarrow\C$ is  linear  and $\phi(x)\in\sigma(x)$ for any $x\in\A$, then $\phi$ is a character. For other alternative proofs of  this result, see \cite{surveyLatif, JR, 10.2307/1998572}. Motivated by this, given a Banach algebra $\A$   and a subalgebra $\B$ of $\A$ with the same unit, the following question arises:

\begin{qu}\label{q2}
Which subset $\Delta(x)$ of $\C$ ensures that a linear function $\phi: \A\longrightarrow\C$ is multiplicative  on $\B$ if $\phi(x)\in\Delta(x)$ for all $x\in\B$.
\end{qu}  Evidently, if $\B$ is a closed subalgebra of a \cs-algebra $\A$, then by the spectral permanence theorem (see \cite[Theorem. 2.1.11]{murphy1990c}) and the GKZ-theorem , it suffices to take $\Delta(x)= \sigma(x)$ for any $x\in\B$.

Let us without further ado, give the outline of the paper. In Section \ref{wellsupported}, after setting up notations and results from the literature, we prove that $\sigma_A(T)$ is a non empty compact subset of $\C$. Next,  if $A$ is well supported  we  give an expression of the $A$-spectrum in terms of pure states. Further, we prove that $\sigma_A(T)=\sigma(PTP, P\vm P)$ for any $T\in\vm^A$. Here $P$   denotes the orthogonal projection onto the range of $A$. An example is provided to show that this equality fails to be true if $A$ is not well supported. Neverthless, it holds when  $T\in\{A\}'$ or if $T$ is algebraic. Here, $\{A\}'$ is the commutant of $A$ in $\vm$. This allows us to show that if $\vm$ and $\vn$ are two von Neumann algebra such that $A\in\vn\subseteq\vm$, then $\sigma_A(T, \vm)=\sigma_A(T, \vn)$ if $T\in\{A\}'$ or if $T$ is algebraic.
In Section \ref{GKZ}  we will apply the results of section \ref{wellsupported} to get some $A$-spectral characterizations. We prove that $\sigma_A(T)$ is finite for any $T\in\vm^A$ if and only if $A$ is in the socle of $\vm$. Note that when $A=I$ we recapture Kaplansky's finite spectrum lemma (\cite[Lemma 7]{kaplansky_1954}).  Also, we show that two  elements $S$ and $T$ satisfying $\sigma_A(TX)=\sigma_A(SX)$ for any $X\in\vm^A$ verify $AS=AT$. Particularly when $\vm$ is prime, the condition $r_A(TX)=r_A(SX)$ for any $X\in\vm^A$ is studied.

Concerning Question  \ref{q2}, we show that if $\phi: \vm^A\longrightarrow\C$ is linear and satisfies $\phi(T)\in\sigma_A(T)$, for any $T\in\vm^A$, then  $\phi$ can be extended to an element of $\sgm$ so that   $\phi(ST)=\phi(S)\phi(T)$ for any $S, T\in\vm^A$, with  $\phi(A)\neq 0$ and $\phi(I)=1$. In particular, if $\vm^A=\vm$, then $\phi$ is a character of $\vm$.
  %In Section \ref{aprox}  we introduce the notion of $A$-approximate point spectrum of elements in $\vm^A$ and investigate its properties.

\section{Further results  on the $A$-spectrum}\label{wellsupported}
In the remainder, we shall denote by $\vm$ a von Neumann algebra acting on a Hilbert space $\h$ and let $A\in\vm$ be a nonzero  positive operator.

If   $\mathcal{S}$ is a subset  of $\lh$, $\mathcal{S}'$ denotes the commutant of $\mathcal{S}$ in $\lh$. Namely,   the collection of all elements in $\lh$ that commute with all elements of $\mathcal{S}$. The double commutant $\mathcal{S}''$ of $\mathcal{S}$ is $\left(\mathcal{S}'\right)'$.
%\[\mathcal{S}:=\{T\in\lh: TS=ST, \forall S\in\mathcal{S}\}.\]
It is also known by the von Neumann double commutant theorem (see \cite[Theorem 4.1.5]{murphy1990c}) that $\vm$ is a von Neumann algebra if and only if $\vm=\vm''$.

 Following \cite[Definition II.3.2.8]{blackadar2006operator}, $A$ is called well-supported if
$\sigma(A)\backslash\{0\}$ is closed (or equivalently $\sigma(A)\subseteq\{0\}\cup[\alpha, \infty)$ for some $\alpha> 0$). In that case the sets $A\vm$ and $\vm A$ are closed in $\vm$ (see e.g. \cite[Proposition II.3.2.11]{blackadar2006operator}). In particular $\RR(A)$ is closed, by \cite{Kulkarni2000ACO}.

For an element $T\in\vm$, $\sigma(T, \vm)$ will denote the spectrum of $T$. When no confusion arises, we shall denote it simply by $\sigma(T)$.

 The following facts about the Moore-Penrose inverse of $A$ will be used later on in the paper: There exists a unique closed densely defined
operator $A^\dagger$ (called the Moore-Penrose inverse of $A$) with domain $D(A^\dagger) = \RR(A)\oplus \RR(A)^\perp=\RR(A)\oplus\NN(A)$ and satisfying the following properties
\[ AA^\dagger h=Ph, \forall h\in D(A^\dagger)\ \text{and}\  A^\dagger A h=Ph, \forall h\in\h.\]
It is well-known that $A^\dag$ exists  in $\lh$ if and only if $\RR(A)$ is closed. For properties and applications of the Moore–Penrose
inverse see \cite{caradus1978generalized, Mbekhta1992ConormeEI}  and the references therein.

A positive linear functional $f$ on $\sm$ is said to be pure if for every positive
functional $g$ on $\vm$ satisfying $g(T)\le f(T)$ for all $T\in\vm^+$, there is a scalar $0\le \alpha\le 1$ such that $g =\alpha f$. The set of pure states on $\vm$ is denoted by $\pM$. See \cite{murphy1990c} for properties about pure states. Similarly, we shall say that a linear functional $\tau\in\sga$ is $A$-pure if $\tau$ has the property that whenever $p$ is a positive linear functional on $\vm$ such that $p\le \tau$ then there is a scalar $\alpha \in [0, 1]$ such that $p =\alpha \tau$. It is easy to see that    \[\pma=\left\{\frac{f}{f(A)}: f\in\pM, f(A)\neq 0\right\}.\]
 %In particular,    for any $h\in\h$ with $\la Ah, h\ra=1$ the functional $T\longmapsto\la Tx, x\ra$ is an $A$-pure state on $\vm$.

Observe that by Lemma \ref{adjoint}, we have $\vm A^{1/2}\subseteq \vm^A$ and $A\vm^A\subseteq A^{1/2}\vm A^{1/2}$.  For an element $T\in\vm^A$,  we define the set $\mathcal{P}_A(T):=\mathcal{P}_A^l(T)\cup \mathcal{P}_A^r(T)$ where
\begin{subequations}\label{psm}
\begin{equation}
\mathcal{P}_A^r(T):=\Big\{\tau\in\pma: \tau(ATS)=\tau(AT)\tau(AS), \forall S\in A^{1/2}\vm A^{1/2}\Big\}   \end{equation}
\begin{equation}
   \mathcal{P}_A^l(T):=\Big\{\tau\in\pma: \tau(AST)=\tau(AT)\tau(AS), \forall S\in A^{1/2}\vm A^{1/2}\Big\}.
   \end{equation}\end{subequations}
Observe that
\begin{equation}\label{psmI}
  \begin{split} \mathcal{P}_I(T)=\Big\{\tau\in\pma: \tau(TS)=\tau(T)\tau(S), \forall S\in  \vm  \
   \text{or}\    \tau(ST)=\tau(T)\tau(S), \forall S\in\vm\Big\}.\end{split}\end{equation}
We shall use repeatedly the following   lemma which is quoted from \cite[Theorem 2.3]{mz2023}.
\begin{lem}
  \label{adjoint} Let $\vm$ be a von Neumann algebra and $A\in\vm$ be positive. Then
     \[\vm^A=\lho\cap\vm=\left\{X\in \vm: X \ \text{has an} \ A^{1/2}-\text{adjoint}\right\}.\]
    \end{lem}
It is worth to note that, by Douglas factorization lemma for operators on $\lh$; (see \cite{Douglas1966}), we have    $T\in\lho$ if and only if $T^*(\RR(A^{1/2})\subseteq\RR(A^{1/2})$. In particular  $T(\mathcal{N}(A))\subseteq \mathcal{N}(A)$ whenever $T\in\lho$.

Also, for our future use we state the following elementary lemma.
\begin{lemma}
  \label{KM} Consider a von Neumann algebra $\vm$  and $A\in\vm$ be positive. We have the following.
  \begin{enumerate}
    \item If $A$ is well supported then $\vm_A=\vm^A$.
    \item Suppose $S\in\vm$ and $f\in\sm$ such that $|f(S)|=\|S\|$. Then $f(ST)=f(TS)=f(S)f(T)$  for any $T\in\lh$.
    \item For any $T\in\vm_A$, we have $\sigma_A^l(T)=\overline{\sigma_A^r(L)}$ where $L\in\vm$ is such that $AT=L^*A$ (resp. $A^{1/2}T=L^*A^{1/2}$ ) . Here $\overline{\sigma_A^r(L)}:=\{\overline{\lambda}: \lambda\in \sigma_A^r(L)\}$ where $\bar{\lambda}$ denotes the conjugate of the complex number $\lambda$.
        \item If $L$ is an $A$-adjoint (resp. $A^{1/2}$-adjoint) of $T$,  then $\sigma_A(T)=\overline{\sigma_A(L)}$.
  \end{enumerate}
 \end{lemma}
\begin{proof} For statements (1) and  (2), see   \cite[Corollary 2.4]{mz2023} and    \cite[Lemma 2.1]{KM2023}.

   Next  and for statement (3), we will give the proof only when $AT=L^*A$ . The other case is similar. Let $\lambda\notin\sigma_A^l(T) $. Then $AS(T-\lambda I)=A=(T-\lambda I)^*S^*A$ for some $S\in\vm_A$. Write $AS=W^*A$ for some $W\in\vm_A$, we get $A(L-\overline{\lambda} I) W=A$. That is $\overline{\lambda}\notin\sigma_A^r(L)$. Whence $\overline{\sigma_A^r(L)}\subseteq\sigma_A^l(T)$. Similarly we can show that $\sigma_A^l(T)\subseteq  \overline{\sigma_A^r(L)}$. Finally, statement (4) follows immediately from statement (3).
\end{proof}
  For a densely defined and bounded operator  $X$, we shall use the  notation $\overline{X}$ for its unique continuous extension to all of $\h$.

Our first theorem of this section follows as:

\begin{thm}
  \label{nempty} Let $\vm$ be a von Neumann algebra and
$P$ be the orthogonal projection onto the range of $A$. The following statements hold.
\begin{enumerate}
  \item For any $T\in \vm^A$, we have $\sigma(PTP, P\vm P)^r\subseteq \sigma_A^r(T)$ and $\sigma(PTP, P\vm P)^l\subseteq \sigma_A^l(T)$.
  \item Let $T\in\vm_A$ and $L\in\vm^A$ such that $A^{1/2}T=L^*A^{1/2}$, then $\|T\|_A=\|PLP\|=\|PL\|$ and $\|L\|_A=\|PTP\|=\|PT\|$.
  \item Let $T\in\vm_A$ and $L\in\vm^A$ such that $A^{1/2}T=L^*A^{1/2}$ and let $\lambda\in\C$ such that $\max\left(\|T\|_A, \|L\|_A\right)<|\lambda|$, then $\lambda I\pm T$ and $\overline{\lambda} I\pm T$ are $A$-invertible. %Here $\overline{\lambda}$ denotes the complex conjugate of $\lambda$.
\end{enumerate}
\end{thm}
\begin{proof}
\noindent
(1)
  Note that $AP=A$ and $\RR(1-P)=\NN(A)$. As $X(\NN(A))$ is a subset of $\NN(A)$ for each $X\in\vm^A$, then $\RR(X(1-P))\subseteq\NN(A)$. Whence $PX(1-P)=0$. Accordingly  $PXP=PX$ and then $PX\in P\vm P$ for each $X\in\vm^A$. Also, recall that $P\vm P$ is a strongly closed  hereditary \cs-subalgebra of $\vm$ with unit $P$ and $\sigma(X, \vm)=\sigma(X, P\vm P)\cup\{0\}$ for each $X\in P\vm P$; see \cite[Theorem 1.6.15]{rickart1960general}.
\\
Let us prove first that $\sigma(PTP, P\vm P)^r\subseteq \sigma_A^r(T)$ and $\sigma(PTP, P\vm P)^l\subseteq \sigma_A^l(T)$. To do so, it suffices to prove that $PXP$ is left (resp. right) invertible in $P\vm P$ whenever $X$ is left (resp. right) $A$-invertible in $\vm^A$.  Pick up an left (resp. right)  $A$-invertible element $X\in\vm^A$.   Then there exists $Y\in \vm^A$ so that $AYX=A$ (resp. $AXY=A$).
Multiplying both sides of this equation by $A^\dagger$, we get $PYX=P$ (resp. $PXY=P$). But since $P^2=P$, $PXP=PX$ and $PYP=PY$ we infer that $PYP(PXP)=P$ (resp. $(PXP) PYP=P$). That is $PXP=PX$ is left (resp. right) invertible in $P\vm P$.
\\ \vskip 0.0001 cm\noindent
(2)  Let $T\in\vm^A$ and note that  there exists  $L\in\vm^A$ so that $A^{1/2} T=L^*A^{1/2}$, by Lemma \ref{adjoint}.
 Observe by \cite[Theorem 3.5]{mabrouk2020} and \cite[Proposition 2.2]{arias2008metric} we have
    \begin{eqnarray*}
      \|T\|_A =\|A^{1/2} T(A^{1/2})^\dag\|&=& \|L^*A^{1/2}(A^{1/2})^\dag\| \\
       &=& \|L^*  P{\vert}_{D\left((A^{1/2})^\dag\right)}\|
       \\ &=& \left\|\overline{L^*  P{\vert}_{D\left((A^{1/2})^\dag\right)}}\right\|.
    \end{eqnarray*}
  The uniqueness of the extension of $ L^*  P{\vert}_{D\left((A^{1/2})^\dag\right)} $ yields that $\overline{L^*  P{\vert}_{D\left((A^{1/2})^\dag\right)}}=L^*P$. This together with the fact that $PL=PLP$ imply that
  \[\|T\|_A=\|L^*P\|=\|PL\|=\|PLP\|.\]
   Finally, since $L\in\vm_A$, $A^{1/2} L=T^*A^{1/2}$  similarly we can show that  $\|L\|_A=\|PTP\|=\|PT\|$.
  \\ \vskip 0.02 cm
(3)  Let $\lambda\in\C$ such that $\|T\|_A<|\lambda|$ and $\|L\|_A<|\lambda|$.  We may replace  $T$ and $L$ by $ \lambda^{-1}T$ and $(\overline{\lambda})^{-1} L$ respectively and assume   $\lambda=1$.  By Theorem \ref{nempty}-(2), we have $\|PLP\|<1$. Then by using  \cite[Theorem 1.2.2]{murphy1990c} on the Banach algebra $P\vm P$ we see that $ P-PLP$ is invertible in $P\vm P$ with
  \begin{equation}\label{inverse}
  \left( P-PLP\right)^{-1}=\sum_{n\ge 0} (PLP)^n=\sum_{n\ge 0} PL^nP. \end{equation}
On the other hand, since $L\in\vm_A$, $A^{1/2} L=T^*A^{1/2}$ and  $\|L\|_A<1$, a similar reasoning entails also that $ P-PTP$ is invertible in $P\vm P$ with
  \begin{equation}\label{inverset}
  \left( P-PTP\right)^{-1}=\sum_{n\ge 0} (PTP)^n=\sum_{n\ge 0} PT^nP. \end{equation}
  Now straightforward computation shows that
  \[A^{1/2}\left(P-PTP\right)^{-1}=\left(\left(P-PLP\right)^{-1}\right)^*A^{1/2}.\] Whence  $\left(P-PTP\right)^{-1}$ and  $\left(P-PLP\right)^{-1}$ are in $\vm$, by Lemma \ref{adjoint}. Further
  \[A(I-T)(P-PTP)^{-1}=A(P-PTP)(P-PTP)^{-1}=AP=A.\]
  Similarly we have $A(P-PTP)^{-1}(I-T)=A$. Hence $I-T$ is $A$-invertible. Finally, a similar reasoning entails also that $I+T$ is $A$-invertible. This completes the proof.
 \end{proof}
 As an immediate consequence of Theorem \ref{nempty}, we get the following result.
\begin{cor}\label{spectreborne}
 For any $T\in \vm^A$ we have $\sigma_A(T)$ is a non empty compact subset of $\C$. %Further $\sigma_A(T)$ is compact if $T\in\vm_A$.
\end{cor}
\begin{proof}
   By \cite[Theorem 1.2.5]{murphy1990c}, we have  $\sigma(PTP, P\vm P)$ is non empty. Hence so is   $\sigma_A(T)$, by Theorem \ref{nempty}-(1). Now, let $T, L\in\vm^A$ such that $A^{1/2}T=L^*A^{1/2}$. Set $\alpha=\max(\|T\|_A, \|L\|_A)$ and note that by Theorem \ref{nempty}-(2),  we have  $|\lambda|\le \alpha$ for any $\lambda\in\sigma_A(T)$. So, it remains to show that  $\sigma_A(T)$ is closed.  Suppose that  $\lambda_0\notin\sigma_A(T)$ then $T_0=\lambda_0I -T$ has an $A$-inverse in $\vm^A$. Pick up an $A$-inverse $S_0$ of $T_0$ and note that we have $P=PS_0T_0=PT_0S_0$. We may write
    \begin{eqnarray*}
      P\left(\lambda I-T\right) &=& \left(\lambda-\lambda_0\right)P +PT_0 \\
        &=& \left(\lambda-\lambda_0\right)PT_0S_0  +PT_0 \\
        &=& PT_0\left(\left(\lambda-\lambda_0\right)S_0  +I\right).
    \end{eqnarray*}
    Let $\beta<\min\left(\frac{1}{\|S_0\|_A}, \frac{1}{\|L_0\|_A}\right)$, where
    $L_0$ is an $A^{1/2}$-adjoint of $S_0 $. If $|\lambda-\lambda_0|<\beta$, then  we can invert the operator $\left(\lambda-\lambda_0\right)S_0  +I$, by Theorem \ref{nempty}-(3). Keeping in mind that $P$ is $A$-invertible, then $P\left(\lambda I-T\right)$ is $A$-invertible, by \cite[
Theorem 2.]{mz2023}. In particular $\left(\lambda I-T\right)$ is $A$-invertible. Hence the open disc in the complex plane of center $\lambda_0$ and radius $\beta$ is contained in $\C\backslash\sigma_A(T)$. Thus $\sigma_A(T)$ is closed and the proof is complete.
\end{proof}
\begin{rem} For $r>0$, let $\overline{D}(0, r)$ be the closed unit disc centered at $0$ with radius $r$.

  By the proof of Corollary \ref{spectreborne}, we have
  \begin{equation}\label{inclusionspectre}
    \sigma_A(T)\subseteq \overline{D}(0,  \alpha\}\ \text{where}\ \alpha=\max(\|T\|_A, \|L\|_A).
  \end{equation}
Also, if $A$ is well-supported, by \cite[Lemma 3.3]{mz2023}, we know that $\sigma_A(T)\subseteq \overline{D}(0, \|T\|_A)$.

If $A$ is not well supported, the next example shows that in general $\|T\|_A\neq \|L\|_A$ and that the inclusion  $\sigma_A(T)\subseteq \overline{D}(0, \|T\|_A)$ may  fails for some $T\in\vm^A$. \end{rem}
\begin{example}
  Let $a_n=\frac{1}{2^n}$   for any $n\ge 0$ and consider the Hilbert space $\h = \ell^2(\N)$ and  the multiplication operator $A$ defined
on $\h$ by $A(x_n) = (a_n^2 x_n)$. Consider   also the right (resp. left) weighted
shift operator $T$ (resp. $L$) given by
\[T(x_n) = \frac{2}{5}\ \left(0, x_0, x_1,\cdots\right)\ (\text{resp.}\ L(x_n) = \frac{1}{5}\left( x_1,  x_2, \cdots\right))\] for all
$(x_n)_{n\ge 0}\in\h$. One can check
easily that $T^*A^{1/2}=A^{1/2}L$. In particular $T, L $ are in $\lho$. Further, since $A$ is injective, then $P=I$ and
\[\|T\|_A=\|L\|= \frac{1}{5}\ \text{and}\ \|L\|_A=\|T\| =\frac2{5},\] by Theorem \ref{nempty}-(2). Whence $\|T\|_A\neq\|L\|_A$. Also, in view of Lemma \ref{KM}-(4) and Theorem \ref{nempty}-(1) we see that
\[r_A(L)=r_A(T)\ge r(L)=\|L\|=\frac{2}{5}.\] Therefore $\sigma_A(T)\nsubseteq\overline{D}(0, \|T\|_A)$.
%any $\lambdain \C$ satisfying $\|T\|_A=\frac{1}{3}<\lambda<\frac{2}{5}$ is in the $A$-spectrum of $T$. Indeed: assume that it not the case. I, we see that $\overline{\lambda}\notin\sigma(L)$. This is a contradiction since $r(L)=\|L\|=$.
\end{example}
The next theorem,
which may be of independent interest, gives an expression of the $A$-spectrum in terms of pure states.   It gives also a refinement of \cite[Theorem 3.6]{mz2023}.

Recall that an approximate unit for a \cs-algebra $\A$ is an increasing net $(U_\alpha)_{\alpha\in\Lambda}$
of positive elements in the closed unit ball of $\A$ such that $X =\displaystyle\lim_\alpha XU_\alpha=\displaystyle\lim_\alpha U_\alpha X=X$  for
all $X\in\A$. Observe that, since $A$ is well supported, all the obtained  results remain true in the frame  of \cs-algebras. For the sake of readability we shall restrict our selves on von-Neumann algebras.
\begin{thm}
  \label{refine}
    Let $\vm$ be a von Neumann algebra. If $A$ is well-supported, then
     \begin{equation}\label{spp}
     \sigma_A(T)= \Big\{g(AT): g\in\mathcal{P}_A(T) \ \Big\},
  \end{equation}
for any $X\in\vm^A$.  In particular if  $A=I$, then
       \begin{equation}\label{spid}
        \sigma(T)=
       \Big\{g(T): g\in\mathcal{P}_I(T) \ \Big\},
       \end{equation} for any $T\in\vm$.
\end{thm}
\begin{proof} First note that $AA^\dag=A^\dag A=P$,  since $A$ is well supported. This together with \cite[Theorem 3.5.1]{blackadar2006operator} entails that
\[P\vm P=A\vm A=A^{1/2}\vm A^{1/2}.\]
Now, set $\Theta=\Big\{g(AT): g\in\mathcal{P}_A(T) \ \Big\}$ and let $\lambda=g(AT)$ for some $g\in\mathcal{P}_A(T)$. If $\lambda\notin\sigma_A(T)$, then  there exists $S\in\vm^A$ so that $A(T-\lambda)S =A=AS(T-\lambda)$. If  $g(ATS)=g(AT)g(AS), \forall S\in A^{1/2}\vm A^{1/2}$. Keeping in mind that $PXP=PX$   for each $X\in\vm^A$, yields that
\begin{eqnarray*}
  g(A)=g(A(T-\lambda )S) &=& g(A(T-\lambda )PSP) \\
    &=& g(ATPSP)-\lambda g(ASP) \\
    &=& g(AT)g(ASP)-\lambda g(ASP)=0
\end{eqnarray*}
  Whence $g(A)=0$ which is impossible. since $g(A)=1$. Similarly, the condition $g(AST)=g(AT)g(AS)$  for any $ S\in A^{1/2}\vm A^{1/2}$  leads up to a contradiction. Hence   $\Theta\subseteq\sigma_A(T)$.

  Now, let $\lambda\in\sigma_A(T)$. By \eqref{sp}, there exists $f\in\sgm$ so that $\lambda=f(AT)$ and  $f(ATS)=f(AT)f(AS)$ for all $S\in\vm^A$ \text{or}
    $f(AST)=f(AT)f(AS)$ for all $S\in\vm^A$. Suppose that
     \begin{equation}\label{right}
     f(ATS)=f(AT)f(AS)\ \text{ for all}\ S\in\vm^A.
     \end{equation} The second case can be handled in a similar way. From \eqref{right} we have  $f\left(A(T-\lambda)S\right)=0$ for all $S\in\vm^A$.  Write $f=\frac{\psi}{\psi(A)}$ for some $\psi\in\sm$ with $\psi(A)\neq 0$, then
    \begin{equation}\label{psi}
      \psi\left(A(T-\lambda)S\right)=0;
    \end{equation}  for all $S\in\vm^A$.
    By Lemma \ref{adjoint}, $A(T-\lambda)$ is in $A^{1/2}\vm A^{1/2}$ and $S A^{1/2}$ belongs to $\vm^A$ for any $S\in \vm$. The continuity of $\psi$ together with \eqref{psi} imply $\psi\left(A (T-\lambda)S\right)=0$
     for all $S\in \cl\left(A^{1/2}\vm A^{1/2}\right)$. Here $\cl\left(A^{1/2}\vm A^{1/2}\right)$ denotes the norm closure of $A^{1/2}\vm A^{1/2}$.
     Hence by \cite[Theorem 3.5.5]{murphy1990c}, there exists a pure state $g\in\pa$ so that $g\left(A (T-\lambda)S\right)=0$
     for all $S\in \cl\left(A^{1/2}\vm A^{1/2}\right)$. Observe that $g(A)\neq 0$ since otherwise $g$ vanishes on $ \cl\left(A^{1/2}\vm A^{1/2}\right)$. Hence, if we put $\tau=\frac{g}{g(A)}$, then $\tau\in\pma$ and $\tau\left(ATS\right)=\lambda\tau\left(AS\right)$
     for all $S\in \cl\left(A^{1/2}\vm A^{1/2}\right)$. Now, $\cl\left(A^{1/2}\vm A^{1/2}\right)$ is \cs-subalgebra of $\vm$ and by \cite[Theorem 3.1.1]{murphy1990c} it admits an approximate identity, say $(U_\alpha)_{\alpha\in\Lambda}$. In particular $\tau\left(ATU_\alpha\right)=\lambda\tau\left(AU_\alpha\right)$ for any $\alpha\in\Lambda$. Taking the limit we get  $\lambda=\tau(AT)$ and then $\lambda\in\Theta$.
\end{proof}
%%%%%%%%%%%%%%%%%%%%%%%%%%%%%
\begin{rem}
  It is clear from Theorem \ref{refine} that
  \begin{equation*}
    \sigma_A^r(T)=\Big\{g(AT): g\in\mathcal{P}_A^r(T) \Big\} \ \text{and}\  \sigma_A^l(T)=\Big\{ g(AT): g\in\mathcal{P}_A^l(T)\Big\}.\end{equation*}
\end{rem}

\begin{cor}\label{cor1} Let  $\vm$ be von Neumann algebra and $A\in\vm$ is well-supported. The following hold.
\begin{enumerate}
  \item If $\vm$ is commutative,  then $\vm^A=\vm$ and
     \begin{equation}\label{spcom}
    \sigma_A(T)=\Big\{\varphi(T): \varphi\ \text{is a character and } \ \varphi(A)\neq 0 \Big\}\subseteq \sigma(T).   \end{equation}
  \item $A$ is $A$-invertible and $\sigma_A(A)=\sigma(A)\backslash\{0\}$.
\end{enumerate}
  \end{cor}
\begin{proof}First observe that if $\vm$ is commutative then any $T\in\vm$ has an $A^{1/2}$-adjoint. Hence $\vm^A=\vm$.
  The equality \eqref{spcom} follows from Theorem \ref{refine} and the fact that pure states on a commutative von Neumann algebras are characters and that  the spectrum of $T$ coincides with the set of values on $T$ of all the characters of $\vm$.   For the second statement,  consider the von Neumann algebra $\mathcal{W}(A)$ generated by $A$. As $A$ is positive $\mathcal{W}(A)$ is commutative. Hence   the $A$-spectral permanence theorem (see \cite{mz2023}) together with \eqref{spcom} entail  that $\sigma_A(A, \vm)=\sigma_A(A, \mathcal{W}(A))=\sigma(A)\backslash\{0\}$. In particular $A$ is $A$-invertible with $A$-inverse $A^\dag$ which is in $\mathcal{W}(A)^A\subset\vm^A$ by \cite[Corollaire 1.8]{Mbekhta1992ConormeEI}.
\end{proof}

%%%%%%%%%%%%%%%%%%%%%%%%%%%%%%%%%%%%%%%%%%%%%%%%%%%%%%%%%%
The next theorem gives another information about the $A$-spectrum when $A$ is well-supported.% In particular it shows that $\sigma_A(T)$ is nonempty for any $T\in\vm^A$.
\begin{thm}
  \label{inclu} Let $\vm$ be a von Neumann algebra and
$P$ be the orthogonal projection onto the range of $A$.   If   $A$ is well-supported then $\sigma_A^l(T)=\sigma^l(PT , P\vm P)$ and $\sigma_A^r(T)=\sigma^r(PT , P\vm P)$ for any $T\in \vm^A$. In particular $\sigma_A(T)=\sigma(PT , P\vm P)$.

\end{thm}
\begin{proof}
 Assume that $A$ is well supported and let $X\in\vm^A$ such that $PXP$ is left invertible in $P\vm P$.
 Then There exists $Y\in\vm$ so that $PYPXP=P$. Whence $AYPXP=A(PYP) X$, since $PXP=PX$ Observe that $PYP\in\vm^A$, since $PYP(\RR(A))\subset\RR(A)$. Therefore $X$ is left $A$-invertible. Similarly we can show that $PXP$ is left invertible in $P\vm P$  implies that $X$ is right $A$-invertible. This completes the proof.
\end{proof}
As a direct consequence of Theorems \ref{refine} and \ref{inclu}, we obtain the following
\begin{cor}
  If $\vm=\mnc$, then $\sigma_A(T)=\sigma(PT, A\mnc A)$ and
  \[\sigma_A(T)= \left\{\tr(QPT):  Q\ \text{is a rank one projection}, \ \tr(QP)=1\right\}.\]

\end{cor}
\begin{proof}
  This follows from the fact that pure sates on $\mnc$ are the functionals of the form $T\in\mnc\longmapsto \tr(QT)$ with $Q$ is a rank one projection.
\end{proof}

\begin{cor}
  \label{xy=yx} Assume that $A$ is well-supported and let $T, S\in\vm^A$ be non-zero elements. The following statements hold.
  \begin{enumerate}
    \item $\sigma_A(ST)\backslash\{0\}=\sigma_A(TS)\backslash\{0\}$. In particular $r_A(TS)=r_A(ST)$
    \item If $ATS=AST=0$, then  \[\sigma_A(T+S)\backslash\{0\}=\left(\sigma_A(T)\cup\sigma_A(S)\right)\backslash\{0\}.\]
    \item If $S^2=S$ and $T^2=T$, then
     \[\sigma_A\left((I-T)(I-S)\right)\backslash\{0, 1\}= \sigma_A(TS)\backslash\{0, 1\}.\]
     \end{enumerate}
\end{cor}
\begin{proof}
   By Proposition \ref{mr1}, we have
  \begin{eqnarray*}
    \sigma_A(TS)\backslash\{0\} &=& \sigma(PTS)\backslash\{0\} \\
      &=& \sigma\left((PTP)(PSP)\right)\backslash\{0\}, \ (\text{since}\  PTP^2 =PT\ \text{and}\ PSP=PS)
      \\ &=&\sigma\left((PSP)(PTP)\right)\backslash\{0\}, \ (\text{since}\  \sigma(XY)\backslash\{0\}=\sigma(YX)\backslash\{0\}, \forall X, Y\in\vm))
 \\&=& \sigma\left(PST \right)\backslash\{0\}=\sigma_A\left(TS\right)\backslash\{0\}. \end{eqnarray*}
 If $AST=ATS=0$ then $PST=(PT)(PS)=(PS)(PT)=0$. Hence by \cite[Lemma 3.2]{ransford_white_2001}, it yields that
 \begin{eqnarray*}
    \sigma_A(S+T)\backslash\{0\} &=& \sigma\left(PT+PS\right)\backslash\{0\}= \left(\sigma(PT)\cup\sigma(PS)\right)\backslash\{0\}=\left(\sigma_A(T)\cup\sigma_A(S)\right)\backslash\{0\}.
    \end{eqnarray*}
\end{proof}

Another consequence is the following proposition.
\begin{prop}\label{mr1}
Let $\vm$ be a von Neumann algebra and
$P$ be the orthogonal projection onto the range of $A$. If $A$ is well-supported and $T, L\in\vm^A$ such that $AL=T^*A$, then
$\sigma_A(T)\backslash\{0\}=\sigma(PT)\backslash\{0\}$,
\begin{equation}\label{inclu1}
 \sigma^r_A(T)\backslash\{0\}=\Big\{g(T): g\in\mathcal{P}_I^r(T)\ \text{and}\ g(P)=1\Big\}\backslash\{0\}
\end{equation} and
\begin{equation}\label{inclu2}
\ \sigma^l_A(T)\backslash\{0\}=\Big\{\overline{g(L)}: g\in\mathcal{P}_I^l(L)\ \text{and}\ g(P)=1\Big\}\backslash\{0\}.\end{equation} In particular, %$\sigma_A(T)\subseteq\sigma(T)\cup\overline{\sigma(L)}\cup\{0\}$ and
if $AT$ is self adjoint we have $\sigma_A(T)\subseteq\sigma(T)$.
\end{prop}
\begin{proof}We have   $\sigma(PTP, P\vm P)=\sigma_A(T)$, by Theorem \ref{inclu}. Whence
\[\sigma_A(T)\backslash\{0\}=\sigma(PTP, P\vm P)\backslash\{0\}=\sigma(PT)\backslash\{0\}, \text{by \cite[Theorem 1.6.15]{rickart1960general}.}
\]  Next,  we show that $\sigma^r_A(T)\backslash\{0\}=\Lambda$ where
\[\Lambda=\Big\{g(T): g\in\mathcal{P}_I^T\ \text{and}\ g(P)=1\Big\}\backslash\{0\}.\] To that end,  recall that $\sigma^r_A(T)\backslash\{0\}=\sigma^r(PT)\backslash\{0\}$ by Theorem \ref{inclu}. Pick up an element $\lambda\in \sigma^r(PT)\backslash\{0\}$. Applying equality \eqref{spid} in Theorem \ref{refine}, then  there is a pure state $g$ on $g $  such that   $\lambda=g(PT)$ and $g(PTS)=g(PT)g(S), \forall S\in\ \vm $.    In particular
\[\lambda=g(PT)=g(PTP)=g(PT)g(P).\]  Hence   $ g(P)=1=\|P\|$, since $\lambda\neq 0$.  By Lemma \ref{KM}-(2), we infer that $g(PS)=g(SP)=g(S)g(P)=g(S)$ for any $S\in\vm$. Therefore $g(P)=1$, $\lambda=\vf(PT)=\vf(P)\vf(T)=\vf(T)$ and
\[g(TS)=g(P)g(TS)=g(PTS)=g(PT)g(S)=g(T)g(S)\] for any $S\in\vm$. This implies that $\lambda\in\Lambda$.

 Conversely, pick up any $\lambda\in\Lambda$. Then $\lambda=g(T)$ with $g\in\mathcal{P}_I^T$ and $ g(P)=1$. Again by Lemma \ref{KM}, we have $g(PS)=g(SP)=g(S)g(P)=g(S)$ for any $S\in\vm$. In particular $\lambda=g(PT)$ and straightforward computations show that $g\in\mathcal{P}_I^{PT}$. Whence $\lambda\in\sigma(PT)\backslash\{0\}$. Finally   equality \eqref{inclu2} follows from Lemma \ref{KM}-(3) and \eqref{inclu1}. The proof is thus complete.
\end{proof}
\begin{cor}\label{corn1}
Let $\vm$ be a von Neumann algebra and  $A\in\vm$ is well-supported, then $r_A(T)\le r(T)$ for any $T\in\vm^A$.
\end{cor}
\begin{proof}
  The result follows from the fact that  $r_A(T)=\displaystyle\max_{\lambda\in\sigma_A^r(T)}|\lambda|$. See \cite[Corollary 3.18]{mz2023}.
\end{proof}
\begin{cor} \label{corn2}Let $A\in\mnc$ be a positive matrix. Then   $P$ is in the centre modulo the
radical of $\mnc^A$. That is $PT-TP$ belongs to the Jacobson radical of $\mnc^A$, for any $T\in\mnc^A$.\end{cor}
\begin{proof}
  Observe that $\mnc^A$ is a closed subalgebra of $\mnc$. Further $\sigma(T, \mnc)$ does not separate
the complex  plane $\C$. Hence by \cite[Corollary 3.2.14.]{aupetit1991primer} we get $\sigma(T, \mnc)=\sigma(T, \mnc^A)$ for any $T\in\mnc^A$. This together with Corollary \ref{corn1} entail that $r(PT, \mnc^A)\le r(T, \mnc^A)$ for any $T\in\mnc^A$. Hence $P$ is in the centre modulo the
radical of $\mnc^A$, by \cite[Theorem 3.1.5]{aupetit1991primer} and \cite[Theorem 3.1]{brevsar2012determining}.
\end{proof}
 \begin{rem}
  In general $\mnc^A$ is not semi-simple. To see why this take $\vm=\mathcal{M}_2(\mathbb{C})$   and let
 $
A =\begin{bmatrix}
1 & 0\\
0 & 0
\end{bmatrix}$.
Then we have
\begin{align*}
\vm_{A}=\left\{\begin{bmatrix}
\alpha & 0 \\
\beta & \gamma
\end{bmatrix}: \,(\alpha, \beta, \gamma)\in\mathbb{C}^3\right\}.
\end{align*}
Using the spectral characterization of the radical in Banach  algebras (see \cite{aupetit1991primer}), we see that the matrix $\begin{bmatrix}
0 & 0\\
1 & 0
\end{bmatrix}$ is in the radical of $\vm^A$. whence  $\mnc^A$ is not semi-simple.
\end{rem}
Observe that a similar reasoning as in the  proof of Corollary \ref{corn2} cannot be applied for general von Neumann algebras, since the equality $\sigma(T, \vm)=\sigma(T, \vm^A)$ may fails for some $T\in\vm$. Nevertheless, we can state the following:
\begin{prop}
 Let $\vm$ be a von Neumann algebra and
$P$ be the orthogonal projection onto the range of $A$. If $A$ is well supported   and $\vm=\vm^A$ then  $P$ is in the centre of $\vm$. \end{prop}
\begin{proof}
  Assume now that $\vm=\vm^A$. Again Corollary \ref{corn1} entails that $r(PT)\le r(T)$ for any $T\in\vm$. Theorem 3.1 of \cite{brevsar2012determining} implies that $P$ is a central projection.
\end{proof}

 \begin{rem}
   If $\vm$ is a factor and $A$ is well supported, then  $\vm=\vm^A$ if and only $A$ is invertible.
   % and since $\lh$ has a trivial center we infer that
   \end{rem}

In view of proposition \ref{mr1} and Theorem \ref{inclu},  natural questions are suggested:
\begin{qu}
If $A$ is not necessarily well-supported, does $\sigma_A(T)=\sigma(PT, P\vm P)$ (resp. $\sigma_A(T)\subset\sigma(T)\cup\{0\}$) for any $T\in\vm^A$?

%Under which conditions  these properties hold  true?
\end{qu}
These inclusions may be  proper in general. This follows from the following example:
 \begin{example}
   Let $(a_n)_{n\in\Z}$ be defined by $a_n = 1$ if $n< 0$ and $a_n =\frac{1}{n!}$ if $n\ge 0$. Let $\h=\ell^2(\Z)$ and
 consider the operators $A$ and $T$, given by
\[A^{1/2}x =\sum_{n\in\Z}a_nx_ne_n,\ \ \text{and}\ \ Tx=\sum_{n\in\Z}x_ne_{n+1},\] for any $x=\sum_{n\in\Z} x_n e_n$. Here ${e_n}$ denotes the standard orthonormal basis of $\h$.

It is well known that $T$ is a unitary operators and $\sigma(T)=\T$, where $\T$  denotes the set of all complex
numbers of modulus one. Note also that $T^{-1}=T^*$ is given by $T^* e_n=e_{n-1}$.  See for instance \cite{Bourhim2003OnTL, Aniello} for more information on shift operators.

Since $a_{n+1}\le a_n$ for any $n\in\Z$, one can see easily that $T\in\lho$ and $\|T\|_A=1$. In particular we can take the operator $L$ defined by $L^*e_n=\frac{a_{n+1}}{a_n} e_{n+1}$ as an $A^{1/2}$-adjoint of $T$. It is clear that $\|L\|_A=1$. On the other hand $T^*\notin\lho$, since
$\frac{\|A^{1/2}T^*e_n\|}{\|A^{1/2}e_n\|}=n$,  for any $n\ge 2$.

We claim that \[\sigma_A(T)=\overline{D}(0, 1):=\{\lambda\in\C: |\lambda|\le 1\}.\]
 Indeed:  since $A$ is injective we see that $\sigma(PTP, P\lh P)=\sigma(T)=\T$. This together with Theorem \ref{inclu} imply that $\T\subset\sigma_A(T)$. Further, since $\|T\|_A=\|L\|_A=1$, it follows from Corollary \ref{spectreborne} that $\sigma_A(T)\subset \overline{D}(0, 1)$.

  Now, it is clear that $0\in\sigma_A(T)$ since $T^*\notin\lho$. If $0<|\lambda|<1$
 then  $\|\lambda T^{-1}\|=\|\lambda T^{*}\|<1$. By \cite[Theorem 1.2.2]{murphy1990c} we have
 \[(T-\lambda)^{-1}=T^{-1}(1-\lambda T^{-1})^{-1}=T^{-1}\sum_{k\ge0}(\lambda T^{-1})^k=T^{*}\sum_{k\ge0}(\lambda T^{*})^k.
\]
If $\lambda\notin\sigma_A(T)$, then $A(T-\lambda)S=AS(T-\lambda)S=A$ for some $S\in\lho$. Accordingly $(T-\lambda)S=S(T-\lambda)S=I$. Whence $S=(T-\lambda)^{-1}=T^{*}\sum_{k\ge0}(\lambda T^{*})^k$ belongs to $\lho$. But
\begin{eqnarray*}
  A^{1/2}(T-\lambda)^{-1}e_n &=&A^{1/2}\left(T^{*}\sum_{k\ge0}(\lambda T^{*})^k\right)e_n
  \\
  &=& \sum_{k\ge0}(\lambda)^k A^{1/2}e_{n-1-k}
   \\
    &=& \sum_{k=0}^{n-1}(\lambda)^k \frac{1}{(n-1-k)!}e_{n-1-k}+\sum_{k\ge n}(\lambda)^k  e_{n-1-k}
\end{eqnarray*}
for $n$ large enough. Hence
\[\frac{\|A^{1/2}(T-\lambda)^{-1}e_n \|}{\|A^2e_n\|}=n!\left\|\sum_{k=0}^{n-1}(\lambda)^k \frac{1}{(n-1-k)!}e_{n-1-k}+\sum_{k\ge n}(\lambda)^k  e_{n-1-k}\right\|\] converges to $\infty$ if $n\to\infty$. This contradicts the fact that $(T-\lambda)^{-1}\in\lho$. Thus $\lambda\in\sigma_A(T)$ as desired.

%If $|\lambda|>1$
 %then $\|\lambda^{-1} T\|<1$. So
%\[(T-\lambda)^{-1}=-\lambda^{-1}(1-\lambda^{-1} T)^{-1}=-\lambda^{-1}\sum_{k\geq 0}(\lambda^{-1}T)^k.\]
%Clearly, $(T-\lambda)^{-1}\in\lho$ \textcolor[rgb]{1.00,0.00,0.00}{DO YOU AGREE WITH THIS} and then $\lambda\notin\sigma_A(T)$.
Hence we have shown that  $\sigma_A(T)=\overline{D}(0, 1)$ which implies that $\sigma(PT, P\vm P)=\T\varsubsetneq \sigma_A(T)$ and    $\sigma_A(T)\nsubseteq\sigma(T)\cup\{0\}$
 \end{example}
%Concerning the segond question We are unable at present to prove or disprove these facts.
Our objective in the forthcoming is to present some situations where the aforementioned questions can be answered in the affirmative.

Using  Corollary \ref{corn2}  we have:
\begin{prop}
  Let $A\in\mnc$ be a positive matrix. Then $\sigma_A(T)\subseteq  \sigma(T)\cup\{0\}$ for any $T\in\mnc^A$.
\end{prop}
\begin{proof}
  Let $R$ be the Jacobson radical of the  Banach algebra $\mnc^A$. By \cite[Theorem 3.1.5]{aupetit1991primer}, we know that  $\mnc^A/R$ is semi-simple and $\sigma(T, \mnc^A)=\sigma(T, \mnc^A/R)$, where $\hat{T}$ is the coset of $T$. Observe that $\hat{
  P}$ is in the centre of $\mnc^A/R$, by Corollary \ref{corn2}. This together with  \cite[Theorem 2.2]{brevsar2012determining} imply that
  \[\sigma(PT, \mnc^A)=\sigma(\hat{P}\hat{T}, \mnc^A/R)\subseteq\sigma(T)\cup\{0\}\ \text{for all} \ T\in\mnc^A.\]
  The proof is thus complete.
\end{proof}
The next theorem tells us that the equality  $\sigma_A(T)=\sigma(PT, P\vm P)$ holds true if:  $T\in\{A\}'$, or  if $T$ is algebraic.
 Here
 \[\{A\}' = \{T \in\lh: TA=AT\}\cap\vm\] is the commutant of $A$ in $\vm$. It is worth observing that $\{A\}'\subset\vm^A$  by Fuglede's theorem.
\begin{thm}\label{com}
Let $\vm$  be a von Neumann algebra  and $A\in\vm$ be positive. The following statements hold.
\begin{enumerate}
  \item $\sigma_A(T)= \sigma(PT, P\vm P)  \ \text{ for any}\ T\in\{A\}'$.  In particular if $\vm$ and $\vn$ are two von Neumann algebras such that $A\in\vn\subseteq\vm$, then
$\sigma_A(T,\vm)=\sigma_A(T,\vn)$  for every $T\in\{A\}'\cap\vn$.
\item If   $T\in\vm^A$ is algebraic, then $\sigma_A(T)= \sigma(PT, P\vm P)$.
%\item If   $0\notin W(T)$  then $\sigma_A(T)= \sigma(PT, P\vm P)$.
\end{enumerate}
\end{thm}
\begin{proof}
 Observe that $\sigma(PT, P\vm P)\subseteq \sigma_A(T)$, by Theorem \ref{inclu}. For the converse assume that $PTP$ is invertible in $P\vm P$. Then $\left(PTP\right)^{-1}$ belongs to the von Neumann subalgebra $ \mathcal{W}(PTP)$ of $P\vm P$ generated by $P, PTP$  and $PT^*P$, the weak closure in $P \vm P$ of the set of complex polynomials in $P, PTP$  and $PT^*P$. Note that $ \mathcal{W}(PTP)$ is also norm closed and $\mathcal{W}(PTP)\subset\vm^A$. Hence $\sigma_A(T)= \sigma(PT, P\vm P)$.

 Next, if $\vm$ and $\vn$ are two von Neumann algebras such that $A\in\vn\subseteq\vm$ and $T\in\{A\}'\cap\vn$, then $\sigma(PT, P\vm P)=\sigma(PT, P\vn P)$ by
  the spectral permanence theorem (see \cite[Theorem. 2.1.11]{murphy1990c}). Whence $\sigma_A(T,\vm)=\sigma_A(T,\vn)$.

  Assume now that   $T\in\vm^A$ is algebraic. Let $\lambda\in \C\backslash\sigma(PTP, P\vm P)$. As $T$ is algebraic then so is $PTP-P$ in the algebra $P\vm A$.  Hence, $ \left(PTP-\lambda P\right))^{-1}$ is a polynomial  in $PTP-P$. Whence $ \left(PTP-\lambda P\right)^{-1}\in\vm^A$.  This completes the proof.

  % Finally, let $T\in\vm^A$ so that $0\notin W(T)$. Then $T$ is invertible and  $\RR(A^{1/2})$ is an invariant subspace for $T^*$. Assume that $\RR(A^{1/2})$ is not invariant under $(T^{-1})^*=(T^*)^{-1}$
\end{proof}
An intriguing sidelight of Theorem \ref{com}
is that the $A$-spectrum does not depends on the commutative von Neumann algebra containing $T$.
\begin{cor} Let $\vm$  be a von Neumann algebra  and $A\in\vm$ be positive. The following statements hold.
\begin{enumerate}
\item If $\vm $ is  commutative, then  $\sigma_A(T)= \sigma(PT, P\vm P)\subseteq\sigma(T)\cup\{0\} $ for any $T\in\vm$. Moreover,
\begin{equation}\label{car}
  \sigma_A(T)\backslash\{0\}=\Big\{g(T): g \ \text{is a character such that}\ g(P)=1\Big\}\backslash\{0\}
\end{equation}
\item $\sigma_A(P)=\{1\}$ and $A$ is $A$-invertible if and only if $A$ is well supported.
  \item $\sigma_A(A)=\sigma(A, \vm)\backslash\{0\}$ if $A$ is well supported  and $\sigma_A(A)=\sigma(A, \vm) $ otherwise.
  \item For any $x, y\in\h$, we have $\sigma_A(x\otimes A^{1/2} y)\subseteq\left\{0, \la x, A^{1/2}y\ra\right\}$.
  \end{enumerate}
\end{cor}
\begin{proof}
  The first statement   follows from Theorem \ref{com}-(1) since $\{A\}'=\vm$. The proof of \eqref{car}   uses the same scheme  as in the proof of Proposition \ref{mr1}. For the second statement, note that $P\in\{A\}'$ and then $\sigma_A(P)=\sigma(P, P\vm P)=\{1\}$.  Items (3) and (4) follows from Theorem \ref{com}-(2).
\end{proof}
\begin{rem}
  Using Lemma \ref{KM}, one can see that $r_A(A)=r(A)=\|A||=\|A\|_A$.
\end{rem}
\begin{cor}  Let $\vm$ be a commutative von Neumann algebra, $A$ a well-supported positive element  in $\vm$ and $T, T_k\in\vm^A$ ($k\in\mathbb{N}$), $\|T_k-T\|_A\to 0$.
Then $\lambda\in\sigma_A(T)$ if and only if there exist points $\lambda_k\in\sigma_A(T_k)$ ($k = 1, 2,\cdots$) such
that $\lambda = \displaystyle\lim_{k\to\infty}\lambda_k$.
    \end{cor}
    \begin{proof}
      If $\lambda\in\sigma_A(T)$, then there exists a character $\phi$ with $\phi(A)\neq 0$ and $\phi(T) =\lambda$. Set $\lambda_k =\phi(T_k)$.
Then by Theorem \ref{refine}, $\lambda_k\in\sigma_A(T_k)$ and $\lambda_k\to\lambda$.

Conversely, let $\lambda_k\in\sigma_A(Tk)$ such that $\lambda_k$ converge to some $\lambda$ in $\C$. By Corollary \ref{cor1}, for each $k$ there exists a character $\phi_k$ such that $\phi_k(A)\neq 0$ and $\phi_k(T_k)=\lambda_k$.  Set $\alpha_k = \phi_k(T)$ which belongs to $\sigma_A(T)$. We have
\[|\lambda-\alpha_k| \le |\lambda-\lambda_k|+|\lambda_k-\alpha_k| = |\lambda-\lambda_k|+|\phi_k(T_k)-\phi_k(T)| \le |\lambda-\lambda_k|+\|T_k-T\|_A\to 0.\]
Thus $\alpha_k\to \lambda$ and $\lambda\in\sigma_A(T)$.
    \end{proof}
   % We close this section by the following result.
    %\begin{prop}
    %  If $T$ and $T^*$ are in $\vm^A$, then $\sigma_A(T)=\sigma(T, \vm)\backslash\{0\}$.
    %\end{prop}
    %\begin{proof}
     % If $T$ and $T^*$ are in $\vm^A$, then $\RR(TA^{1/2})\subseteq\RR(A^{1/2})$ and $\RR(T^*A^{1/2})\subseteq\RR(A^{1/2})$ \cite[Lemma 2.2]{Vosough2017SolutionsOT}
    %\end{proof}
 %   We close this section with the following.
   % \begin{thm}
   % If $\|T\|_A<1$ then $I-\lambda$ is $A$-invertible. In particular   $\sigma_A(T)$ is a nonempty compact subset of $\C$, for any $T\in\vm^A$.
    %\end{thm}
    \section{$A$-spectral characterizations}
    \label{GKZ} This section presents  several applications of the results of Section \ref{wellsupported}. The first one characterizes von Neumann algebra for which $\sigma_A(T)$ is finite for any $T\in\vm^A$.

    A well known result of Kaplansky (see \cite{kaplansky_1954} or \cite{aupetit1991primer}), asserts that a semi-simple Banach algebra $\A$ is finite dimensional if and only if $\sigma(X)$ is finite for any $X\in\A$. % See \cite{aupetit1991primer} for further considerations.
     The purpose of the next theorem is to show that this result need not to be true if the spectrum is replaced by the $A$-spectrum and if $A$ is not one to one. %To do so, we shall need the following lemma quoted from \cite{mz2023}.

    Recall that the socle of $\vm$, denoted by $\soc(\vm)$, is the sum of all minimal left (or right) ideals
of $\vm$. Note that the socle of any von Neumann algebra exists. Further, $\soc(\vm)$ is a two sided  ideal of $\vm$ and is consisting of
all elements $T\in\vm$ satisfying $\sigma(TX)$ is  finite for all $X\in\A$, and that all its
elements are algebraic. We refer the
reader to \cite{Aupetitbook2,  aupetit1991primer} for more details.

We have the following.
\begin{thm}
  \label{kapl} Let $\vm$ be a von Neumann algebra and $A\in\vm$ be positive. The following are equivalent.
  \begin{enumerate}
    \item $\sigma_A(T)$ is finite for any $X\in\vm^A$.
    \item $A$ is in the socle  of $\vm$.
  \end{enumerate}
\end{thm}
\begin{proof}
  Assume that $\sigma_A(T)$ is finite for any $X\in\vm^A$. Since $XA$ is in $\vm^A$ for any $X\in\vm$, we get $\sigma_A(AX)$ is finite for any $X\in\vm$. From Theorem  \ref{inclu} and the fact that $PA=A$, we infer that $\sigma(AX)$ is finite for any $X\in\vm$. Whence $A$ is in the socle of $\vm$, by \cite[Corollary 2.9]{Aupetit1996}. For the converse if $A$ is in the socle  of $\vm$. Again \cite[Corollary 2.9]{Aupetit1996} entails that $A\vm A$ is a finite dimensional von Neumann algebra containing $A$. Hence $P\in A\vm A$ and then $P\vm P\subseteq A\vm A$. But, since $PA=AP=A$, we infer that $P\vm P= A\vm A$ is finite dimensional.
  Therefore, by Proposition \ref{mr1}, we get $\sigma_A(X)\backslash\{0\}=\sigma(PX)\backslash\{0\}=\sigma(PXP)\backslash\{0\}$ is finite for any $X\in\vm^A$. The proof is thus complete.
\end{proof}
Another variant of Theorem \ref{kapl} is the following.
\begin{thm}
  \label{kap2} Let $\vm$ be a von Neumann algebra and $A\in\vm$ be positive.  If  $\sigma_A(T)$ has at most $n$ points for every $T\in\vm^A$, then $\dim A\vm A\le n^2$.

\end{thm}
\begin{proof}
 Note that by Theorem \ref{kapl}. $A$ is in the socle of $\vm$ and then $A$ is well supported. In particular $A\vm A=A^{1/2}\vm A^{1/2}$ is a finite dimensional  von Neumann  algebra. Further  using Theorem  \ref{inclu} and the hypothesis, we see that $\sigma(AX)=\sigma(A^{1/2} X A^{1/2})$ has   at most $n$ points for every $T\in\vm^A$.  The together with \cite[Lemma 1]{Behncke1973NilpotentEI} allow us to conclude.
\end{proof}
The next theorem characterizes elements $S$ and $T\in\vm^A$ satisfying $\sigma_A(SX)=\sigma_A(TX)$ for any $X\in\vm^A$. To do so, we shall need the following lemma.
\begin{lem}\label{eqs}
  If $\vm$ is von Neumann algebra and $S, T\in\vm$ satisfy $\sigma(SX)\backslash\{0\}=\sigma(TX)\backslash\{0\}$ for all $X\in\vm$, then
$S=T$.
\end{lem}
\begin{proof}
The proof is an adaptation of the one of  Theorem 2.6 in \cite{brevsar2012determining}
and will be omitted
\end{proof}
\begin{thm}
  \label{MA=M}
  Let $\vm$ be a von Neumann algebra and $A\in\vm$ be a well supported. If two elements $S, T\in\vm^A$ satisfy $\sigma_A(SX)=\sigma_A(TX)$ for all $X\in\vm^A$, if and only if
$AS=AT$.
\end{thm}
\begin{proof}
Assume that $AS=AT$ and then $A^{1/2}S=A^{1/2}T$. Observe that $\mathcal{P}_A^T= \mathcal{P}_A^S$. To see why this let $\tau\in\mathcal{P}_A^T$ such that $\tau(ATX)=\tau(AT)\tau(AX)$, for all $X\in A^{1/2}\vm A^{1/2}$ then $\tau\in \mathcal{P}_A^T$. If   $\tau(AXT)=\tau(AT)\tau(AX)$, for all $X\in A^{1/2}\vm A^{1/2}$. Since $A^{1/2}X=L_XA^{1/2}$ for some $L_X\in\vm^A$ then
\begin{eqnarray*}
  \tau(AS)\tau(AX)=\tau(AT)\tau(AX)=\tau(AXT) &=& \tau(A^{1/2}L_XA^{1/2}T) \\
    &=& \tau(A^{1/2}L_XA^{1/2}S)=\tau(AXS).
\end{eqnarray*} Hence $\tau\in \mathcal{P}_A^T$ and then
   $\mathcal{P}_A^T\subseteq \mathcal{P}_A^S$. Similarly, we have $\mathcal{P}_A^S\subseteq \mathcal{P}_A^T$. This together with Theorem \ref{refine} entail that $\sigma_A(S)=\sigma_A(T)$. Since the condition $AS=AT$ implies $ATX=ASX$ for any $X\in\vm^A$, by a similar reasoning one can prove that $\sigma(SX)\backslash\{0\}=\sigma(TX)\backslash\{0\}$ for all $X\in\vm$.

   For the converse assume that $\sigma_A(SX)=\sigma_A(TX)$ for all $X\in\vm^A$. Using Proposition \ref{mr1} we get $\sigma(PSX)\backslash\{0\}=\sigma(PTX)\backslash\{0\}$ for all $X\in\vm^A$. Keeping in mind that $AP=PA=A$ and that $XA\in\vm^A$ for any $X\in\vm$, it follows from \cite[Lemma 3.1.2]{aupetit1991primer} that $\sigma(ASX)\backslash\{0\}=\sigma(ATX)\backslash\{0\}$ for all $X\in\vm^A$. Lemma \ref{eqs} implies that $AS=AT$. The proof is thus complete. \end{proof}

   For the condition $r_A(SX)=r_A(TX)$ for all $X\in\vm^A$, we have the following.
   \begin{thm}Let $A$ be a well supported element of $\vm$.
     Assume further that $\vm$ is a prime  von Neumann algebra and let $S, T\in\vm^A$. Then $r_A(SX)\le r_A(TX)$ for all $X\in\vm^A$  if and only if  there exists a scalar $\alpha\in\C$ such that $|\alpha|\le 1$ and $AS=\alpha AT$.
 \label{ZKHT}  \end{thm}
   \begin{proof}
     It i s clear from Theorem \ref{MA=M}  that if $AS=\alpha AT$ for some scalar $\alpha\in\C$ such that $|\alpha|\le 1$ that $r_A(SX)\le r_A(TX)$ for all $X\in\vm$. For the converse, by a similar reasoning as above we see that the condition $r_A(SX)\le r_A(TX)$ for all $X\in\vm$ implies that $r(ASX)\le r(ATX)$ for all $X\in\vm$. By \cite[Theorem 3.7]{brevsar2012determining} the result follows.
   \end{proof}

  \begin{rem}
    Assume that $A$ is not well supported. If $S$ and $T$ are in $\{A\}'$, then the condition $\sigma_A(SX)=\sigma_A(TX)$ for all $X\in\vm^A$, will imply that $AS=AT$. Indeed: Corollary \ref{com} imply that $\sigma(ASX)\backslash\{0\}=\sigma(ATX)$ for any $X\in\{A\}'$. Since $\{A\}'$ is a \cs-algebra then $AS=AT$ as desired.
   \end{rem}

If $\phi:\vm^A\longrightarrow\C$ is multiplicative,  then $\phi(P)=1$. Therefore  $\phi(T)\in\sigma_A(T)$ for every $T \in\vm^A$. We finish this paper with the following theorem which shows that the condition $\phi(T)\in\sigma_A(T)$ for every $T \in\vm^A$ is necessary for $\phi$ to be multiplicative on $\vm^A$.
 The proof is carried out along the same line of the proofs made in \cite{10.2307/1998572}; (see also \cite[Theorem 1.1]{JR}).
  % We acknowledge that the proof is strongly inspired
%by  the proof of \cite{10.2307/1998572} and \cite[Theorem 1.1]{JR}. %Note that continuity of $\phi$ is not assumed.
\begin{thm}\label{GKZ-A}
   Let $\phi: \vm_A\longrightarrow\C$ be a linear map. If  $\phi(T)\in\sigma_A(T)$ for any $T\in\vm^A$, then   $\phi(TS)=\phi(T)\phi(S)$ for any $S, T\in\vm^A$. In particular if $A$ is in the center of $\vm$, then $\phi$ is a character.
 \end{thm}
 \begin{proof} Suppose that $\phi: \vm_A\longrightarrow\C$ is a linear functional satisfying $\phi(T)\in\sigma_A(T)$ for all $T\in\vm^A$.
   Note that $\phi(I)=1$ since $\phi(I)\in\sigma_A(I)=\{1\}$.
\\
   Let $T\in\vm^A$. Given $n\ge1$ and $\lambda\in\C$ and denote by $p$ the polynomial $p(\lambda)=\phi\left((\lambda I-T)^n\right)$. Let
$\lambda_1, \cdots, \lambda_n$ be the roots of $p$. Since $0 = p(\lambda_i) = \phi\left((\lambda_i I- T)^n\right)\in\sigma_A\left((\lambda_i I- T)^n\right)$,  then $(\lambda_i I- T)^n$ is not $A$-invertible. As then $(\lambda_i I- T)$ is also not $A$-invertible. Thus$\lambda_i\in\sigma_A(T)$ for $i = 1,\cdots n$. Keeping in mind that $\phi(I)=1$, we can write
\begin{eqnarray*}
  p(\lambda) =\displaystyle\prod_{i=1}^{n} (\lambda-\lambda_1) &=& \lambda^n -\lambda^{n-1}\sum_{i=1}^{n}\lambda_i+\lambda^{n-2}\sum_{i\neq j}\lambda_i\lambda_j+\cdots \\
    &=& \lambda^n-n\phi(T)\lambda^{n-1}+\frac{n(n-1)}{2}\phi(T^2)\lambda^{n-2}+\cdots.
\end{eqnarray*}
Identifying terms,  we get
\[
   n\phi(T)=\sum_{i=1}^{n}\lambda_i\ \text{and}\ \sum_{i\neq j}\lambda_i\lambda_j=\frac{n(n-1)}{2}\phi(T^2).\]
It follows, in particular, that
\[\left(n\phi(T)\right)^2 - n(n-1)\phi(T^2)=\sum_{i=1}^{n}\lambda_i^2.\]
Whence
\begin{equation}\label{gk}
  \left|n^2\phi(T)^2  - n(n-1)\phi(T^2)\right|\le n\, r_A(T)^2.
\end{equation}
Note that $r_A(T)$ is finite by Theorem \ref{spectreborne}. Dividing through by $n^2$ in \eqref{gk} and letting $n \to\infty$, we obtain $\phi(T^2)=\phi(T)^2$. Replacing $T$ by $T+S$, yields
\[\phi(ST+TS)=2\phi(S)\phi(T), (T, S\in\vm^A).\]
Finally, assume to the contrary that there exist
$S, T\in\vm^A$ such that $\phi(ST)\neq \phi(S)\phi(T)$. Replacing $S$ by $\alpha S+\beta I$ for some suitably  scalars $\alpha$ and $\beta$ if necessary, we may suppose that $\phi(S) = 0$ and
$\phi(ST) = 1$. Since  $\phi(ST)+\phi(TS)=2\phi(S)\phi(T)$, then $\phi(TS) =-1$. Now, observe that
$\phi(S)\phi(TST)=0$.  On the other hand
\[ 2\phi(S)\phi(TST)=\phi(STST)+\phi(TSTS)=\phi(ST)^2+\phi(TS)^2 = 2.\]
This contradiction finishes  the proof.
 \end{proof}
 \begin{cor}
   Let $\phi: \vm_A\longrightarrow\C$ be a linear map. If $A$ is well supported and $\phi(T)\in\sigma_A(T)$ for any $T\in\vm^A$, then $\phi$ is the restriction of a state $f\in\sm$ such that $\phi(A)\neq 0$ and $\phi(TS)=\phi(T)\phi(S)$ for any $S, T\in\vm^A$.
 \end{cor}
 \begin{proof}
Note that $\phi(TS)=\phi(T)\phi(S)$ for any $S, T\in\vm^A$,  by Theorem \ref{GKZ-A}. Further,    $\phi(A)$ is not zero since $\phi(A)\in\sigma_A(A)$ and $\sigma_A(A)=\sigma(A)\setminus\{0\}$, by Corollary \ref{cor1}. Finally using Corollary \ref{corn1}, we see that
  $|\phi(T)|\le r_A(T)\le r(T)\le \|T\|$ for every $T\in\vm^A$. By Hahn-Banach Theorem there exists $f$  in $\sm$ so that $f=\phi$  on $\vm^A$.
 \end{proof}
% When $A$ is not well-supported we can state the following theorem which answers positively question \ref{q2} when $\vm$ is commutative.
%\begin{thm} If $\vm$ is a commutative von Neumann algebra, then any linear map $\phi: \M\longrightarrow\C$ such that $\phi(T)\in\sigma(PT), (T\in\M)$ is a character.
%\end{thm}
%\begin{proof}
  % The proof follows from Theorem \ref{com} and uses  similar argumnent as in that of Theorem \ref{ZKHt}. Details are omitted.
%\end{proof}

\vskip 0.03 cm
\noindent
\textbf{Competing interests.} The authors declare that they have no conflict of
interest.
\end{document}